\documentclass[11pt,a4paper]{article}

\usepackage[T1]{fontenc}
\usepackage[latin1]{inputenc}
\usepackage{amssymb,latexsym}
\usepackage[all]{xy}
\usepackage{color}
\usepackage[normalem]{ulem}
\usepackage{a4}
\usepackage{rotfloat}

\newcommand{\DeclareMathOperator}[2]{\newcommand{#1}{\mathop{\mathrm{#2}}
                                     \nolimits}}

\newif\iflabel
\labelfalse
\newcommand{\Label}[1]{\iflabel\ifmmode\makebox[0pt][l]{{\color{red}[#1]}}
                       \else\marginpar{{\color{red}[#1]}}\fi\fi\label{#1}}

\DeclareMathOperator{\diag}{diag}

\DeclareMathOperator{\rk}{rk}
\DeclareMathOperator{\Syl}{Syl}

\newcommand{\Z}{{\mathbb Z}}

\newcommand{\N}{{\mathbb N}}
\newcommand{\F}{{\mathbb F}}

\newcommand{\ti}{\tilde}
\newcommand{\la}{\langle}
\newcommand{\ra}{\rangle}
\newcommand{\im}{{\rm im}}

\newcommand{\GAP}{{\sf GAP}}

\newtheorem{theorem}{Theorem}
\newtheorem{lemma}[theorem]{Lemma}

\newenvironment{proof}{\par\vskip-\lastskip\vskip\topsep
                       \noindent{\it Proof.}\vadjust{\nobreak}\quad
                       \begingroup\divide\topsep3\divide\itemsep3
                       \divide\partopsep3\divide\parskip3
                       \divide\parsep3}
                      {\ifvmode\penalty10000\hbox to\hsize{\hfil$\Box$}
                       \else\parfillskip0pt\widowpenalty10000\hfil$\Box$
                       \fi\par\vskip 1.5ex\endgroup}

\floatstyle{plain}
\floatname{algorithm}{Algorithm}
\newfloat{algorithm}{tbp}{lop}

\author{Ren\'e Hartung}
\title{Solving linear equations over finitely generated abelian groups}
\date{\today}

\begin{document}
\maketitle
\begin{abstract}
  We discuss various methods and their effectiveness for solving linear
  equations over finitely generated abelian groups. More precisely, if
  $\varphi\colon G\to H$ is a homomorphism of finitely generated abelian
  groups and $b\in H$, we discuss various algorithms for checking whether
  or not $b\in \im\varphi$ holds and if so, for computing a pre-image
  of $b$ in $G$ together with the kernel of $\varphi$.
\end{abstract}

%%%%%%%%%%%%%%%%%%%%%%%%%%%%%%%%%%%%%%%%%%%%%%%%%%%%%%%%%%%%%%%%%%%%%%%%%%%%
\section{Introduction}
Solving linear equations over finitely generated abelian groups is an
important tool for computing $H^2(G,A)$ of a given $G$-module $A$.
Therefore algorithms for solving linear equations are important in
extension theory of groups,~see Section~8.7 of~\cite{HEO05}. For
extending the Small Groups Library, see e.g.~\cite{BEO02}, we need
effective algorithms especially for this problem. In this note, we
discuss various methods and we propose an effective algorithm which
performs well on a broad range of randomly chosen linear equations.

%%%%%%%%%%%%%%%%%%%%%%%%%%%%%%%%%%%%%%%%%%%%%%%%%%%%%%%%%%%%%%%%%%%%%%%%%%%%
\section{Preliminaries}
Let $G$ and $H$ be finitely generated abelian groups and let
$\varphi\colon G\to H$ be a homomorphism. Further let $b\in H$ be given.
A \emph{linear equation over $G$} asks whether or not $b\in\im\varphi$
holds and, if so, for computing a pre-image of $b$ in $G$ together with
the kernel of $\varphi$.

\section{Standard methods for polycyclic groups}
Since the groups $G$ and $H$ are finitely generated and abelian, both
are polycyclic.  Therefore the methods for homomorphisms of polycyclic
groups described in~\cite[Section~3.5.2]{Eick01} apply. We briefly
summarize these methods here for completeness.\medskip

The homomorphism $\varphi\colon G\to H$ naturally corresponds to
the subgroup $U_\varphi = \{ (g^\varphi,g) \mid g\in G \}$ of the
direct product $H\times G$. Given polycyclic sequences ${\mathcal
G}$ and ${\mathcal H}$ of $G$ and $H$ respectively, the algorithms
in~\cite{Eick01} allow to compute an induced polycyclic sequence for
$U_\varphi$ with respect to ${\mathcal H} \times{\mathcal G}$. This
sequence allows to read off the kernel of $\varphi$ and can be used to
decide whether or not $b\in\im\varphi$ holds. 
\begin{lemma}[Eick, 2001]
  Let $\varphi$ be a group homomorphism and ${\mathcal U}$ be an induced
  polycyclic sequence for $U_\varphi$ with respect to ${\mathcal H}\times
  {\mathcal G}$. Then ${\mathcal U}$ has the form $((\bar u_1,u_1),\ldots,(\bar
  u_s,u_s))$ with $u_i\in G$ and $\bar u_i\in H$. Let $t$ be maximal 
  with $\bar u_t\neq 1$. 
  \begin{enumerate}\addtolength{\itemsep}{-0.75ex}
  \item Then $(\bar u_1,\ldots,\bar u_t)$ is an induced polycyclic sequence
        for the image of $\varphi$ with respect to ${\mathcal H}$.
  \item Then $(u_{t+1},\ldots,u_s)$ is an induced polycyclic sequence
        for the kernel of $\varphi$ with respect to ${\mathcal G}$.
  \end{enumerate}
\end{lemma}
\begin{proof}
  For a proof, we refer to~\cite[p.36]{Eick01}.
\end{proof}
This method is implemented in the {\scshape Polycyclic} package of
\GAP, see~\cite{Polycyclic}. However, for finitely generated abelian
groups, we can do much better as we will show in the following.

%%%%%%%%%%%%%%%%%%%%%%%%%%%%%%%%%%%%%%%%%%%%%%%%%%%%%%%%%%%%%%%%%%%%%%%%%%%%
\section{Solution by solving linear Diophantine equations}\label{sec:DioSNF}
In this section we describe a more effective algorithm relying on linear
algebra only. Let $G$ be a finitely generated abelian group. Then $G$
decomposes into its torsion subgroup $T(G)$ and a free abelian subgroup
so that $G \cong \Z^\ell \times T(G)$ holds. Further, the torsion subgroup
$T(G)$ decomposes into its $p$-Sylow subgroups so that we may identify
\[
  G = \Z^{\ell} \oplus \Syl_{p_1}(T(G)) \oplus\cdots\oplus \Syl_{p_k}(T(G)),
\]
where $\Syl_{p_i}(T(G))$ denotes the $p_i$-Sylow subgroup of
$T(G)$. Similarly, we can identity
\[
  H = \Z^{\ell'} \oplus \Syl_{p_1}(T(H)) \oplus\cdots\oplus \Syl_{p_k}(T(H)).
\]
Every homomorphism $\varphi\colon G\to H$ induces a homomorphism
$\varphi_i = \varphi|_{\Syl_{p_i}(T(G))}$ of the $p_i$-Sylow
subgroup. Assume that $\Syl_{p_i}(T(H))\cong \Z_{p_i^{\alpha_{i1}}}
\oplus\cdots\oplus \Z_{p_i^{\alpha_{in_i}}}$ holds. Then, with respect
to an independent generating set of $G$, the equation $x^\varphi = b$
translates to equations of the form
\begin{equation}\label{eqn:ModSys}
% \begin{array}{ccccccccccccc}
%   a_{11} x_1 &+& \cdots &+& a_{1s} x_s &+& a_{1,s+1}x_{s+1} &+& \cdots &+& 
%   a_{1t} x_t &\equiv& b_1 \mod p_1^{\alpha_{11}} \\
%   \vdots && \ddots && \vdots && \vdots && \ddots && \vdots && \vdots \\
%   a_{m1} x_1 &+& \cdots &+& a_{ms} x_s &+& a_{m,s+1}x_{s+1} &+& \cdots &+& 
%   a_{mt} x_t &\equiv& b_m \mod p_m^{\alpha_{mn_m}} \\
% \end{array}
  \left(\begin{array}{cc}
    A_t & B \\ 
    0   & A_s
  \end{array}\right) \, \xi = \beta \pmod{[p_1^{\alpha_{11}},\ldots,p_m^{\alpha_{mn_m}},0,\ldots,0]}
% \left(\begin{array}{c}
% x_{1}\\ \vdots\\x_{\ell}
% \end{array}\right) = 
% \left(\begin{array}{c}
% b_1 \\ \vdots \\ b_\ell
% \end{array}\right) \pmod{\textrm{\color{red}hallo welt!}} 
\end{equation}
where $a_{i1} x_1 + \cdots + a_{in} x_n = b_i \pmod{0}$ denotes the
Diophantine equation, the sub-matrix $A_t$ is a block diagonal matrix describing
the homomorphisms $\varphi_1,\ldots,\varphi_k$, while $A_s$ incorporates
the action of $\varphi$ on the free abelian subgroup $\Z^\ell \cong
G/T(G)$.\medskip

For solving~(\ref{eqn:ModSys}), we introduce free variables, one for
each generator of the torsion subgroup $T(G)$. Thereby we obtain a linear
system of Diophantine equations of the form
\begin{equation}\label{eqn:DioSys}
 \left(\begin{array}{ccc}
  A_t & B & D \\
   0  & A_s & 0 
 \end{array}\right)\,\hat\xi = \beta
%\left(\begin{array}{c}
% x_{1}\\ \vdots\\x_{\ell}\\ t_1\\ \vdots\\ t_n
% \end{array}\right) = 
% \left(\begin{array}{c}
% b_1 \\ \vdots \\ b_\ell
% \end{array}\right) 
\end{equation}
where $D = \diag(p_1^{\alpha_{11}},\ldots,p^{\alpha_{mn_m}})$ is a
diagonal matrix describing the modular equations of the form $a_{i1}x_1 +
\cdots + a_{i\ell} x_\ell \equiv b_i \pmod{p^{\alpha_{ij}}}$.

%%%%%%%%%%%%%%%%%%%%%%%%%%%%%%%%%%%%%%%%%%%%%%%%%%%%%%%%%%%%%%%%%%%%%%%%%%%%
\subsection{Solving linear Diophantine equations}
The linear system of Diophantine equations in~(\ref{eqn:DioSys})
can be solved using the Smith normal form; see~\cite{Laz96}. Recall
that, for every $A\in\Z^{m\times n}$ there are unimodular matrices
$L\in{\rm SL}(m,\Z)$ and $R\in{\rm SL}(n,\Z)$ such that $LAR$ is a
diagonal matrix $D=\diag(d_1,\ldots,d_r,0,\ldots,0)$ whose non-zero
entries $d_1,\ldots,d_r$ satisfy $d_{i}\mid d_{i+1}$ for every $ i <
r=\rk(A)$.\medskip

Therefore the linear system $Ax = b$ of linear Diophantine
equations in~(\ref{eqn:DioSys}) is equivalent to the equations
$Dy=Lb$ and $x=Ry$. Write $c=Lb$. Then the system $Dy=c$ has
an integral solution if and only if $d_i \mid c_i$ holds for
every $1\leq i\leq r$ and $c_i=0$ otherwise. If the equation
$Dy = c$ has an integral solution, then these are given by $y =
\left({c_1}/{d_{1}},\ldots,{c_r}/{d_r},t_1,\ldots,t_{n-r}\right)$ with
$t_1,\ldots,t_{n-r}\in\Z$. Every solution $x$ to the linear system
of Diophantine equations $Ax = b$ is then easily obtained as $x =
Ry$.\medskip

It is well known,~\cite{HM97}, that computing the Smith normal form of
an integral matrix is computationally hard due to the unavoidable growth
of the intermediate matrix entries. Even though there are improvements
in special cases, see or~\cite[Section~9.3]{HEO05} for an
overview, we cannot expect the algorithm in this section to be practical 
in general. Especially for groups with a large torsion-subgroup,
the problem grows significantly by forming the system of linear
Diophantine equations in~(\ref{eqn:DioSys}). Therefore reducing to
a minimal generating set for the torsion subgroup first might improve this
algorithm. Note that especially finite groups are a problem here.\medskip

We can use a different and more effective approach for finitely generated
abelian groups which first considers the linear system of Diophantine
equations arising from the torsion-free part in $G$ and $H$. Afterwards,
our approach solves the equations over the torsion subgroups with varying
the right-hand-sides with respect to the solutions of the torsion-free
part.

%%%%%%%%%%%%%%%%%%%%%%%%%%%%%%%%%%%%%%%%%%%%%%%%%%%%%%%%%%%%%%%%%%%%%%%%%%%%
\section{Solving linear equations over finite abelian groups}\label{sec:FAG}
In the remainder we consider linear equations over finite abelian
groups only. Let $G$ and $H$ be finite abelian groups. Suppose
that $p_1,\ldots,p_n$ are prime numbers such that $|G| = p_1^{e_1}
\cdots p_n^{e_n}$ and $|H| = p_1^{f_1} \cdots p_n^{f_n}$ with
$e_i,f_i\in\N\cup\{0\}$. Since every homomorphism $\varphi\colon G\to H$
induces a homomorphism $\varphi_i\colon\Syl_{p_i}(G)\to\Syl_{p_i}(H)$
of the $p_i$-Sylow subgroups, we may restrict to the case that $G$
and $H$ are abelian $p$-groups. More precisely, if we identify
\[
  G = \Syl_{p_1}(G)\times\cdots\times\Syl_{p_n}(G)\quad\textrm{and}\quad
  H = \Syl_{p_1}(G)\times\cdots\times\Syl_{p_n}(G),
\]
then we can decompose $\varphi$ into homomorphisms
$\varphi_i\colon\Syl_{p_i}(G)\to\Syl_{p_i}(H)$ with
$\varphi|_{\Syl_{p_i}(G)} = \varphi_i$. This yields that $\varphi$
decomposes as $\varphi_1 \times \cdots \times \varphi_n$, where
$\varphi_1\times\cdots\times \varphi_n$ acts diagonally by
\[
  (g_1,\ldots,g_n) ^ {\varphi_1\times\cdots\times\varphi_n} 
  = ( g_1^{\varphi_1},\ldots, g_n^{\varphi_n}).
\]
Computing pre-images and the kernel of $\varphi$ can be done independently
for the $p_i$-Sylow subgroups. In particular, solving linear equations
over finite abelian groups split into independent computations for the
Sylows subgroups and hence, can easily be parallelized.

%%%%%%%%%%%%%%%%%%%%%%%%%%%%%%%%%%%%%%%%%%%%%%%%%%%%%%%%%%%%%%%%%%%%%%%%%%%%
\section{Solving linear equations over abelian $p$-groups}
Let $G$ and $H$ be abelian $p$-groups and let $\varphi\colon G\to
H$ be a homomorphism. We identify $G = \Z_{p^{e_1}}\times \cdots
\times \Z_{p^{e_n}}$ and $H = \Z_{p^{f_1}}\times \cdots \times
\Z_{p^{f_m}}$ with $e_1\leq \ldots\leq e_n$ and $f_1\leq \ldots \leq
f_m$. Further let $\{g_1,\ldots,g_n\}$ and $\{h_1,\ldots,h_m\}$ be
independent generating sets of $G$ and $H$ with $|g_i|=p^{e_i}$ and
$|h_i|=p^{f_i}$, respectively. Then we can represent $\varphi$ by an
$m$-by-$n$ matrix $A=(a_{ij})_{i,j}$ where $(a_{1j},\ldots,a_{mj}) =
(\alpha_1,\ldots,\alpha_m)$ whenever $g_j^\varphi = h_1^{\alpha_1}\cdots
h_m^{\alpha_m}$ holds.\medskip

Computing a pre-image of $b=h_1^{b_1}\cdots h_m^{b_m}\in H$ together
with the kernel of $\varphi$ is equivalent to solving the linear system of
modular equations
\begin{equation} \label{eqn:ModSysI}
  \begin{array}{ccccccc}
    a_{11} x_1 &+& \cdots &+& a_{1n} x_n &\equiv& b_1 \pmod{p^{f_1}}\\
     \vdots    && \ddots && \vdots &\vdots&\vdots\\
    a_{m1} x_1 &+& \cdots &+& a_{mn} x_n &\equiv& b_m \pmod{p^{f_m}}
  \end{array}
\end{equation}
We discuss the following methods for solving~(\ref{eqn:ModSysI}):
\begin{enumerate}\addtolength{\itemsep}{-0.75ex}
\item solving a linear system of Diophantine equations,
\item using the method for polycyclic groups described in~\cite{Eick01},
\item using a modular analog of Smith normal form in the case that $G$ and $H$ are both homocyclic,
\item using the method in (C) and lifting solutions recursively,
\item lifting a solution over $\F_p$ recursively with the Hensel lemma.
\end{enumerate}
In Section~\ref{sec:Apps} we show the application of these algorithms
to various randomly chosen equations.

%%%%%%%%%%%%%%%%%%%%%%%%%%%%%%%%%%%%%%%%%%%%%%%%%%%%%%%%%%%%%%%%%%%%%%%%%%%%
\subsection{Solving linear equations over homocyclic $p$-groups}
\label{sec:Homoc}
Let $G$ and $H$ be homocyclic $p$-groups; that is, $G\cong \Z_{p^\ell}
\times\cdots\times \Z_{p^\ell}$ for some $k$ copies of $\Z_{p^\ell}$.
We generalize the algorithm of Section~\ref{sec:DioSNF} to solve linear
equations over homocyclic groups.\medskip

Let $\varphi\colon G \to H$ be a homomorphism and let $g\in H$ be
given. Then the endomorphism $\varphi$ is represented by an $m$-by-$n$
matrix $A=(a_{ij})_{i,j}$ while $g\in H$ translates to its corresponding
exponent vector $(b_1,\ldots,b_m)$. For computing all solutions to the
system $Ax \equiv b \pmod{p^\ell}$, we use the following
modular analog of the Smith normal form:
\begin{lemma}\label{lem:MSNF}
  Let $A\in(\Z_{p^\ell})^{m\times n}$ be given. Then there exist matrices
  $L\in(\Z_{p^\ell})^{m\times m}$ and $R\in (\Z_{p^\ell})^{n\times n}$,
  which are invertible modulo $p^{\ell}$, so that $LAR$ is a diagonal
  matrix $D=\diag(d_1,\ldots,d_k,0,\ldots,0)$ whose non-zero entries
  $d_1,\ldots,d_k$ satisfy $d_i\mid d_{i+1}$ for each $1\leq i<k$.
\end{lemma}
\begin{proof}
  We give a constructive proof for this lemma.  For a positive integer
  $n = a\cdot p^\ell$ with $\gcd(a,p)=1$, we denote by $\nu_p(n)=\ell$
  the \emph{$p$-evaluation of $n$}. Choose indices $1\leq i\leq m$
  and $1\leq j\leq n$ so that
  \[
    \nu_p(a_{ij}) = \min\{ \nu_p(a_{\iota\kappa}) \mid 1\leq\iota\leq m,
    1\leq\kappa\leq n\}
  \]
  By permuting the rows and columns of $A$, we may assume that
  $a_{ij}=a_{11}$ holds. Then $a_{11} = a\cdot p^{\nu_p(a_{11})}$
  for some integer $a$ with $\gcd(a,p)=1$. In particular, the integer
  $a$ is invertible modulo $p^{\ell}$.  Multiplying the first row
  by the inverse $\alpha$ of $a$ yields that the entries $a_{i1}$,
  with $2\leq i\leq m$, are all divisible by $\alpha a_{11}\equiv
  p^{\nu_p(a_{11})}\pmod{p^\ell}$. Similarly, the entries $\alpha a_{1j}$,
  with $2\leq j\leq n$, are all divisible by $\alpha a_{11}$.  Therefore,
  we can find matrices $L_1$ and $R_1$ so that
  \[
    L_1 A R_1 = \left( \begin{array}{cccc}
    p^{\nu_p(a_{11})} & 0 & \cdots & 0 \\
    0 &&&\\
    \vdots& &\ti A&\\
    0 &&&\\
    \end{array}\right)
  \]
  holds. Continuing with the $(m-1)\times (n-1)$ sub-matrix $\ti A$
  recursively would finally yield matrices $L$ and $R$ so that $LAR =
  D$ is diagonal. A permutation of the rows and columns of $D$ would
  then give the divisibility claimed above.
\end{proof}
The algorithm of Section~\ref{sec:DioSNF} now readily generalizes to
an algorithm for solving the linear equation $Ax\equiv b\pmod{p^\ell}$:
By Lemma~\ref{lem:MSNF}, there exist matrices $L$ and $R$ so that
\[
  LAR=\diag(d_1,\ldots,d_k,0,\ldots,0).
\]
Therefore the linear equation $Ax\equiv b\pmod{p^\ell}$ translates to
the systems $Dy\equiv Lb\pmod{p^\ell}$ and $x\equiv Ry\pmod{p^\ell}$. It
remains to check whether or not the diagonal system $Dy \equiv Lb
\pmod{p^\ell}$ has a solution and, if so, to determine all these
solutions. This is a straightforward application of elementary number
theory.

%\begin{lemma}\label{lem:CompMSNF}
%  Let $A\in(\Z_{p^\ell})^{m\times n}$ be given. Then the complexity
%  of computing the modular Smith normal form in Lemma~\ref{lem:MSNF}
%  is {\bf TODO}.
%\end{lemma}
%\begin{proof}
%  Start counting!
%% \begin{itemize}\addtolength{\itemsep}{-1ex}
%% \item Finding pivot element: $n\cdot m-1$ comparisons 
%% \item normalizing the first row: $n$ multiplications 
%% \item eliminating the first colum (including the modification of $L$)
%% \item eliminating the first row (including the modification of $R$)
%% \item {\bf Gcd-Computations?}
%% \end{itemize}
%\end{proof}
Clearly for solving the linear system $Ax \equiv b\pmod{p^\ell}$ the
explicit computation of $L$ is not necessary.  Instead we can apply the
row operations directly to the right-hand-side $b$ which is cheaper in
general. %This yields a complexity of {\bf TODO}.

%%%%%%%%%%%%%%%%%%%%%%%%%%%%%%%%%%%%%%%%%%%%%%%%%%%%%%%%%%%%%%%%%%%%%%%%%%%%
\subsection{Solving linear equations using the block structure}\label{sec:Block}
The algorithm of Section~\ref{sec:Homoc} generalizes to a method
for solving a linear system of equations over arbitrary abelian
$p$-groups as we will describe in the following.\medskip

Let $G$ and $H$ be arbitrary abelian $p$-groups. Further let
$\varphi\colon G\to H$ be a homomorphism and let $g\in H$
be given. We decompose the groups $G$ and $H$ into direct products
\[
  G = (\Z_{p^{e_1}})^{g_1}\times\cdots\times (\Z_{p^{e_\ell}})^{g_\ell} 
  \quad\textrm{and}\quad
  H = (\Z_{p^{e_1}})^{h_1}\times\cdots\times (\Z_{p^{e_\ell}})^{h_\ell}
\]
with $e_1<\ldots<e_\ell$ and $g_i,h_i\in\N\cup\{0\}$. Clearly, in 
the descending chains of subgroups
\[
  G \geq p^{e_1}G \geq \ldots \geq p^{e_\ell} G = \{0\}
  \quad\textrm{and}\quad
  H \geq p^{e_1}G \geq \ldots \geq p^{e_\ell} H = \{0\}.
\]
the factors $p^{e_i}G / p^{e_{i+1}} G$ and $p^{e_i}H/p^{e_{i+1}}H$ are
homocyclic of rank $g_{i+1}+\cdots+g_\ell$ and $h_{i+1}+\cdots+h_\ell$,
respectively. Every homomorphism $\varphi\colon G\to H$ maps the
subgroup $p^{e_i}G$ to $p^{e_i} H$ and therefore induces a homomorphism
$\varphi^{(i)}\colon G/p^{e_{i+1}}G\to H/ p^{e_{i+1}}H$
Let $\iota_i$, $\delta_i$, $\kappa_i$, and $\varepsilon_i$ denote the
natural homomorphisms so that the diagram 
\[
  \xymatrix{
   G \ar@/^5pt/[rd]^{\iota_{i+1}}\ar@/_10pt/[rdd]_{\iota_i} \ar@/^10pt/[rrr]^\varphi &   &  & H\ar@/_5pt/[ld]_{\kappa_{i+1}}\ar@/^10pt/[ldd]^{\kappa_{i}} \\
   & G/p^{e_{i+1}}G\ar[d]^{\delta_i}\ar[r]^{\varphi^{(i+1)}} &H/p^{e_{i+1}}H\ar[d]^{\varepsilon_i}&\\
   & G/p^{e_i}G \ar[r]^{\varphi^{(i)}} & H/p^{e_i}H &
  }
\]
commutes.  Since the quotients $G/p^{e_1}G$ and $H/p^{e_1}H$ are
both homocyclic, the algorithm of Section~\ref{sec:Homoc} computes
an $x_1\in G$ so that $x_1^\varphi-b \in \ker\kappa_1$ holds together
with a generating set for some $K_1 \subseteq G$ with $K_1^{\iota_1}
= \ker\varphi^{(1)}$.  In the remainder of this section, we lift these
solutions recursively to a solution of $x^\varphi = b$. The following 
lemma lifts the special solution $x_i$.
\begin{lemma}\label{lem:LifSpS}
  The solution $x_i$ lifts to a solution $x_{i+1}$ if and only
  if there exists $\lambda\in G$ such that $\lambda^{\iota_i} \in
  \ker\varphi^{(i)}$ with $x_{i+1} = x_i - \lambda$ and
  \begin{equation}\label{eqn:NecCon}
      (x_i^\varphi-b)^{\kappa_{i+1}}
    = \lambda^{\varphi\kappa_{i+1}} =
    \lambda^{\iota_{i+1}\varphi^{(i+1)}}.
  \end{equation}
\end{lemma} \begin{proof}
  Let $\lambda\in G$ be as above. Then 
  \[
    (x_{i+1}^\varphi - b )^{\kappa_{i+1}}
  = ((x_i-\lambda)^\varphi - b)^{\kappa_{i+1}} = x_i^{\varphi\kappa_{i+1}}
  - \lambda^{\varphi\kappa_{i+1}} - b^{\kappa_{i+1}} = 0,
  \] 
  and hence $x_{i+1}^\varphi - b$ lifts the solution $x_i$.  Assume that
  the solution $x_i$ lifts to the solution $x_{i+1}$; that is, we both
  have $x_{i+1} ^ \varphi - b \in \ker\kappa_{i+1}$ and $x_i^\varphi - b
  \in \ker\kappa_i$. Recall that $\ker\kappa_i = p^{e_i}H \geq p^{e_{i+1}}
  H = \ker\kappa_{i+1}$. This yields that $(x_{i+1}-x_i)^\varphi \in
  \ker\kappa_i$ and hence
  \[
    0 = ( x_{i+1} - x_i ) ^ {\varphi\kappa_i} = (x_{i+1} - x_i ) ^
    {\iota_i\varphi^{(i)}}.
  \]
  Therefore $(x_{i+1} - x_i)^{\iota_i} \in \ker\varphi^{(i)}$.
\end{proof}
It remains to lift the kernel of the linear equation.  Let $K_i\subseteq
G$ be a pre-image of $\ker\varphi^{(i)}$ in $G$. Then a generating set for
$K_i$ is easily obtained from the generating sets of $\ker\varphi^{(i)}$
and $\ker\iota_i$. Suppose that $K_i = \la k_1,\ldots, k_\ell\ra$
holds. Since we have that $\iota_i\varphi^{(i)} = \varphi\kappa_i$,
it follows that
\[
  K_i^{\varphi\kappa_{i+1}} = 
  K_i^{\iota_{i+1}\varphi^{(i+1)}} 
  \leq \ker\varepsilon_i 
  = p^{e_i}H / p^{e_{i+1}} H.
\]
It suffices to check whether or not $x_i^\varphi-b \in \la
k_1^{\varphi\kappa_{i+1}}, \ldots, k_\ell^{\varphi\kappa_{i+1}}\ra$ holds.
Since $K_i^{\varphi\kappa_{i+1}}\leq p^{e_i}H/p^{e_{i+1}}H$, the latter
condition is equivalent to solve an equation over the homocyclic group
$p^{e_i}H/p^{e_{i+1}}H$. A solution to this latter system yields a lift
of the solution $x_i$ as described in Lemma~\ref{lem:LifSpS}. Furthermore,
the following lemma outlines the lift of $K_i$ to $K_{i+1}$.
\begin{lemma}\label{lem:LifKer}
  Let $K_i\leq G$ be given so that $K_i^{\iota_i} = \ker\varphi^{(i)}$
  holds. Then it holds that
  \[
    K_{i+1} 
    = \la \{k\in K_i\mid k^{\varphi\kappa_{i+1}}=0\}
      \cup \{  p^{e_{i+1}-e_i}\, k\mid k\in K_i\} \ra.
  \]
\end{lemma}
\begin{proof}
  Let $k\in K_i$ be so that $k^{\varphi\kappa_{i+1}} = 0$. Then, as
  $\varphi\kappa_{i+1} = \iota_{i+1} \varphi^{(i+1)}$, it holds that
  $k^{\iota_{i+1}\varphi^{(i+1)}} = 0$ and therefore, $k^{\iota_{i+1}}
  \in \ker\varphi^{(i+1)}$.  Write $\Delta p = p^{e_{i+1}-e_i}$ and let
  $k\in K_i$ be given. Then it follows that
  \[  
    (\Delta p\, k)^{\iota_{i+1}\varphi^{(i+1)}} 
  = (\Delta p\, k)^{\varphi \kappa_{i+1}}
  \]
  and, since $k^{\varphi}\in p^{e_i}H=\ker\kappa_i$, we also have that
  \[
    (\Delta p\,k)^{\varphi} = \Delta p\, k^{\varphi} \in p^{e_{i+1}}H =
    \ker\kappa_{i+1}.
  \]
  This yields that $L = \Delta p\, K_i$ is contained in $K_{i+1}$ and,
  as $k^\varphi\in\ker\kappa_{i+1} \leq \ker\kappa_i$, it follows that
  $K_{i+1}\leq K_i$.  Hence, every element $g\in K_{i+1}\setminus L$
  can be written as $g=ab$ with $b\in L$ and $a\in K_{i+1}$ with $aL
  \neq L$. Then we get
  \[
    0 = g ^ {\iota_{i+1}\varphi^{(i+1)}} 
      = (ab) ^ {\varphi\kappa_{i+1}}
      = a ^ {\varphi\kappa_{i+1}} \, b^{\varphi\kappa_{i+1}}
  \]
  where $b^{\varphi\kappa_{i+1}} = 0$. Hence the element $a\in K_i$ 
  satisfies $a^{\varphi\kappa_{i+1}} = 0$.
\end{proof}
Note that the generating set of $K_{i+1}$ defined in
Lemma~\ref{lem:LifKer} may contain redundancies. These significantly
affect the complexity of the algorithm in Section~\ref{sec:Homoc}.

%%%%%%%%%%%%%%%%%%%%%%%%%%%%%%%%%%%%%%%%%%%%%%%%%%%%%%%%%%%%%%%%%%%%%%%%%%%%
\subsection{On a number theoretical approach}
In this section we describe a number theoretical approach
which avoids the redundancies introduced in the lifting of
the kernels in Section~\ref{sec:Block}. Our overall strategy
for solving~(\ref{eqn:ModSys}) is an induction on the exponents
$f_1,\ldots,f_m$. More precisely, we solve the linear system $Ax\equiv
b\pmod p$ over the finite field $\F_p$ with Gaussian elimination. Then
we lift the obtained solutions recursively by applying the Hensel
lemma. First we only consider an abelian $p$-group $G$ and an endomorphism
$\varphi\colon G\to G$.

%%%%%%%%%%%%%%%%%%%%%%%%%%%%%%%%%%%%%%%%%%%%%%%%%%%%%%%%%%%%%%%%%%%%%%%%%%%%
\subsubsection{Solving endomorphic equations over finite $p$-groups}
\label{sec:HenselEndo}
Let $G$ be a finite $p$-group and let $\varphi\colon G\to G$ be an
endomorphism. Further let $b\in G$ be given. We describe an algorithm
for solving the linear equation $x^\varphi = b$ or, equivalently, the
linear system of modular equations
\begin{equation}\label{eqn:LSME}
  \begin{array}{ccccccc}
    a_{11} x_1 &+& \cdots &+& a_{1m} x_m &\equiv& b_1 \pmod{p^{e_1}} \\
      \vdots   & & \ddots & & \vdots     & & \vdots \\
    a_{m1} x_1 &+& \cdots &+& a_{mm} x_m &\equiv& b_m \pmod{p^{e_m}}.
  \end{array}
\end{equation}
Note that the matrix $A = (a_{ij})_{1\leq i,j\leq m}$ in~(\ref{eqn:LSME})
satisfies the condition
\begin{equation}\label{eqn:DivCond}
  p^{e_i-e_{\min\{i,j\}}}\mid a_{ij}\quad\textrm{ for each }
  1\leq i,j\leq m,
\end{equation}
as it corresponds to the endomorphism $\varphi$;
see also~\cite{Gal08}. Denote the linear system
of equations in~(\ref{eqn:LSME}) by $Ax \equiv
b\pmod{[p^{e_1},\ldots,p^{e_m}]}$.\smallskip

Clearly, using Gaussian elimination, we can find all solutions to
the linear system $Ax \equiv b \pmod{[p,\ldots,p]}$ efficiently.
Every solution $x = (x_1,\ldots,x_m)$ to the system over the finite
fields $\F_p$ has the form $x = \xi_0 + \xi_1 t_1 + \cdots + \xi_r
t_r$ where, for each $1\leq i\leq r$, it holds that $A\xi_i \equiv 0
\pmod{[p,\ldots,p]}$ and $t_i \in\{0,\ldots,p-1\}$.\smallskip

The overall idea for solving~(\ref{eqn:LSME}) is to lift these solutions
simultaneously by keeping the coefficients $t_1,\ldots,t_r$ in the
finite field $\F_p$.  This yields that, if $|\ker\varphi| = p^r$,
we obtain $r$ independent homogeneous solutions $\xi_1,\ldots,\xi_r$.
For $\ell\in\{1,\ldots,e_m\}$ and $k = \min\{i \mid \ell\leq e_i\}$,
we assume that all solutions to the linear system $Ax \equiv b
\pmod{[p^{e_1},\ldots,p^{e_{k-1}},p^{\ell},\ldots,p^{\ell}]}$ are given by
\begin{equation}\label{eqn:SolsToLift}
  x = \xi_0 + t_1\xi_1 + \cdots + t_r\xi_r,
\end{equation}
where $t_i\in\{0,\ldots,p-1\}$ for each $1\leq i\leq
r$. Applying the Hensel lemma, we lift the solutions
in~(\ref{eqn:SolsToLift}) to solutions to $Ax \equiv b
\pmod{[p^{e_1},\ldots,p^{e_{k-1}},p^{\ell+1},\ldots,p^{\ell+1}]}$. For
this purpose, we consider the modular equations
\[
  \begin{array}{ccccccc}
    a_{k1} x_1 &+& \cdots &+& a_{km} x_m &\equiv& b_k \pmod{p^{\ell+1}} \\
     \vdots    & & \ddots & & \vdots     & &    \vdots          \\
    a_{m1} x_1 &+& \cdots &+& a_{mm} x_m &\equiv& b_m \pmod{p^{\ell+1}}.
  \end{array}
\]
By condition~(\ref{eqn:DivCond}), the solutions $x_i \pmod{p^{e_i}}$,
for each $1\leq i<k$, do not need to be considered anymore. Similar to
Section~\ref{sec:Homoc}, we can find a matrix $L\in \Z^{m\times m}$,
which is invertible modulo $p^{\ell+1}$, so that $\ti A = LA$ is the
matrix of an equivalent system of modular equations so that $\ti A =
(\ti a_{ij})_{k\leq i\leq m, 1\leq j\leq m}$ still satisfies
\[
  p^{e_i-e_{\min\{i,j\}}}\mid \ti a_{ij}\quad\textrm{ for each }
  k\leq i \leq m\textrm{ and }1\leq j< k,
\]
and the sub-matrix $( \ti a_{ij} )_{k\leq i,j\leq m}$ is
upper-triangular. Assume that all solutions to $A x \equiv
b \pmod{[p^{e_1},\ldots,p^{e_{k-1}}, p^{\ell}, \ldots, p^{\ell},
p^{\ell+1}, \ldots, p^{\ell+1}]}$ are given \emph{uniquely} by
\begin{equation}\label{eqn:Sols}
  x = \xi_0 + t_1 \xi_1 + \cdots + t_r \xi_r
\end{equation}
with $t_i\in\{0,\ldots,p-1\}$; that is, the solutions
$x_m,x_{m-1},\ldots,x_{n-1}$ modulo $p^\ell$ are already lifted to
solutions modulo $p^{\ell+1}$.  We show how to lift the solutions $x_n =
x_n(t_1,\ldots,t_r)$ modulo ${p^{\ell}}$ modulo $p^{\ell+1}$. Consider
the modular equation
\[
  a_{l1} x_1 + \cdots + a_{lk} x_k + a_{ln} x_n + \cdots + a_{lm} x_m 
  \equiv b_l \pmod{p^{\ell+1}}
\]
for some $k\leq l\leq m$. This equation can be considered
a polynomial $\ti f\in\Z[x_n]$ with zeros modulo $p^\ell$
given by $x_n = x_n(t_1,\ldots,t_r)$ for any choice of
$t_1,\ldots,t_r\in\{0,\ldots,p-1\}$. More precisely, we define $\ti f(x_n)
= f(x_n) - b_l$ where
\[
  f(x_n) = a_{l1} x_1 + \cdots + a_{lk} x_k + a_{ln} x_n + \cdots + a_{lm} x_m.
\]
The Hensel lemma now applies to the polynomial $\ti f$ and its zeros
$x_n(t_1,\ldots,t_r)$:  Let $\ti f'(x_n) = \ti a_{ln}$ denote the formal
derivation of $\ti f$. Using the linearity of the polynomial $f$, we
obtain the following lemma which gives a condition for a unique lift of
the zeros $x_n = x_n(t_1,\ldots,t_r)$.
\begin{lemma}\label{lem:UniqueLift}
  If $\ti f'(x_n) \not \equiv 0 \pmod p$ holds, then the solutions $x_n =
  x_n(t_1,\ldots,t_r)$ lift uniquely to $x_n + t\cdot p^\ell$ where
  $t\in\{0,\ldots,p-1\}$ is given by
  \begin{eqnarray} \label{eqn:UniqLift}
    t &\equiv& - (\ti a_{ln})^{-1}\:\left(\frac{f(\xi_0)-b_h}{p^\ell} 
    + \frac{f(\xi_1)}{p^\ell}\,t_1 + \cdots + \frac{f(\xi_r)}{p^\ell}\,t_r 
    \right) \pmod p.
  \end{eqnarray}
\end{lemma}
\begin{proof}
  The proof follows immediately from the Hensel lemma.
\end{proof}
Recall that a solution $x_n(t_1,\ldots,t_r)$ modulo $p^\ell$ may not
lift to a solution modulo $p^{\ell+1}$. The following lemma gives a
sufficient and necessary condition for such a lift.
\begin{lemma}\label{lem:NonUni}
  If $\ti f'(x_n) \equiv 0 \pmod p$ holds, then the
  solutions $x_n = x_n(t_1,\ldots,t_r)$ lift if and only if
  the coefficients $t_1,\ldots,t_r\in\{0,\ldots,p-1\}$ satisfy
  \begin{equation}\label{eqn:CondLift}
    \frac{f(\xi_0)-b_h}{p^\ell} + \frac{f(\xi_1)}{p^\ell}\,t_1 
    + \cdots + \frac{f(\xi_r)}{p^\ell}\,t_r \equiv 0 \pmod p.
  \end{equation}
  If this is the case, then the solutions $x_n = x_n(t_1,\ldots,t_r)$ 
  lift to $x_n+t_{r+1}\,p^\ell$ for each $t_{r+1}\in\{0,\ldots,p-1\}$.
\end{lemma}
\begin{proof} 
  The proof follows immediately from the Hensel lemma. 
\end{proof}
The linearity of~(\ref{eqn:CondLift}) allows to eliminate a free
variable $t_s$, say, if this equation is not trivially satisfied; that
is, at least one coefficient does not vanish modulo $p$. More precisely,
writing $\gamma_i = {f(\xi_i)}/{p^\ell}$, for each $1\leq i\leq r$,
and $\gamma_0 = {(f(\xi_0)-b)}/{p^\ell}$, we may have that $f(\xi_s)
\not \equiv 0 \pmod{p^{\ell+1}}$ but $f(\xi_{s+1}) \equiv \ldots \equiv
f(\xi_r) \equiv 0 \pmod{p^{\ell+1}}$. Then we may restrict
\begin{equation} \label{eqn:ElimTs}
  t_s \equiv -\gamma_s^{-1} \left(
  \gamma_0+\gamma_1 t_1 + \cdots + \gamma_{s-1} t_{s-1} \right) \pmod p.
\end{equation}
The following lemma determines an independent subset of solutions
of~(\ref{eqn:Sols}) which lift by Lemma~\ref{lem:NonUni}.
\begin{lemma}
  If $f(\xi_s)\not\equiv 0\pmod{p^{\ell+1}}$ holds
  while $f(\xi_{s+1}) \equiv \ldots \equiv f(\xi_r) \equiv 0
  \pmod{p^{\ell+1}}$, then the combinations of the elements
  \begin{equation}\label{eqn:ModSols}
    \ti \xi_i = \left\{ \begin{array}{cl}
    \xi_i - \gamma_s^{-1} \gamma_i \xi_s,&\textrm{ for each }1\leq i\leq s-1\\[0.75ex]
    \xi_i,&\textrm{ for each } s+1\leq i \leq r.
    \end{array}\right.
  \end{equation}
  are $p^{r-1}$ solutions amongst~(\ref{eqn:Sols}) which satisfy
  $f(x_n)\equiv 0\pmod{p^{\ell+1}}$.
\end{lemma}
\begin{proof}
  It is easy to see that for any element $\ti\xi_i$ in~(\ref{eqn:ModSols})
  it holds that $f(x_n)\equiv 0\pmod{p^{\ell+1}}$. Thus it remains to prove that
  these elements are independent modulo $[p^{e_1},\ldots,p^{e_k},p^{\ell},
  \ldots, p^{\ell+1},p^{\ell+1},\ldots, p^{\ell+1}]$. For this purpose, 
  assume that we are given $t_i,u_i\in\{0,\ldots,p-1\}$ such that
  \[
    \ti\xi_0 + \sum_{i\neq s} t_i \ti\xi_i \equiv 
    \ti\xi_0 + \sum_{i\neq s} u_i\ti\xi_i \pmod{[\ldots,p^{\ell+1},\ldots]}
  \]
  holds. Then, by construction, we have that 
  \begin{eqnarray*}
    &&\xi_0 + \sum_{i\neq s} t_i \xi_i -\gamma_s^{-1}(\gamma_0 + \gamma_1t_1+\cdots+\gamma_{s-1} t_{s-1})\xi_s\\
    &\equiv& \xi_0 + \sum_{i\neq s} u_i\xi_i - \gamma_s^{-1}(\gamma_0+\gamma_1u_1+\cdots+\gamma_{s-1}u_{s-1})\xi_s\pmod{[\ldots,p^{\ell+1},\ldots]}
  \end{eqnarray*}
  By Equation~(\ref{eqn:ElimTs}), there exist $\delta,\varepsilon\in\Z$ so that 
  \begin{equation}\label{eqn:01}
    \begin{array}{rcl}
    t_s &=& -\gamma_s^{-1}(\gamma_0 + \gamma_1t_1 + \cdots +
    \gamma_{s-1}t_{s-1}) + \delta p  \\[0.7ex]
    u_s &=& -\gamma_s^{-1}(\gamma_0 + \gamma_1 u_1+ \cdots
    +\gamma_{s-1} u_{s-1}) + \varepsilon p.
    \end{array}
  \end{equation}
  Thus we obtain that 
  \[
    \xi_0 + \sum_{i=1}^r t_i \xi_i + \delta p\xi_s 
    \equiv \xi_0 + \sum_{i=1}^r u_i\xi_i + \varepsilon p\xi_s 
    \pmod{[\ldots,p^{\ell+1},\ldots]}.
  \]
  and, in particular,
  \begin{equation}\label{eqn:Cond}
    \xi_0 + \sum_{i=1}^r t_i \xi_i + \delta p\xi_s 
    \equiv \xi_0 + \sum_{i=1}^r u_i\xi_i + \varepsilon p\xi_s 
    \pmod{[\ldots,p^{\ell},\ldots]}.
  \end{equation}
  If $p \xi_s$ vanishes modulo $[\ldots,p^\ell,\ldots]$, the
  claim follows from the independence of $\xi_1,\ldots,\xi_r$
  in~\ref{eqn:Sols}. Otherwise, as $f(p\xi_s) \equiv 0 \pmod{p^{\ell}}$
  holds, the element $p\xi_s$ can be written uniquely as
  $p\xi_s = \alpha_{s+1}\xi_{s+1} +\cdots+ \alpha_r\xi_r$ with
  $\alpha_i\in\{0,\ldots,p-1\}$. Note that the the elements $\xi_i$, with
  $s<i\leq r$, have a weight strictly less than $\xi_s$; i.e. if $w(\xi)
  = \min\{ i \mid p^i\xi \equiv 0 \pmod{[\ldots,p^\ell,\ldots]}\}$
  denotes the \emph{weight} of $\xi$, then, by construction, we have
  that $w(\alpha_i) \leq w(\xi_s)$. Suppose that $w(\alpha_i) = w(\xi_s)$
  holds for some $s<i\leq r$. Then we obtain that
  \[
    \alpha_{s+1}\,p^{a-1}\,\xi_{s+1} + \cdots + \alpha_r\,p^{a-1}\,\xi_r
    = p^a \xi_s \equiv 0 \pmod{[\ldots,p^\ell,\ldots]}
  \]
  which contradicts the independence of $\xi_{s+1},\ldots,\xi_r$ of
  weight $p^{a}$ if at least one $\alpha_i$ does not vanish. Thus, by
  induction on the weight, we obtain that Equation~(\ref{eqn:Cond}) 
  implies that $t_1 = u_1,\ldots,t_s = u_s$, and additionally 
  \[
    t_i + \delta \alpha_i \equiv u_i + \varepsilon\alpha_i \pmod p.
  \]
  By Equation~(\ref{eqn:01}), this yields that $\delta = \varepsilon$
  and therefore, by the choice of the $t_i$'s and $u_i$'s we finally
  obtain $t_i = u_i$, for each $1\leq i<s$ or $s<i\leq r$.
\end{proof}
We summarize the algorithm in~\ref{alg:Hensel}.
\begin{algorithm}[ht]\label{alg:Hensel}
\begin{center}
\begin{minipage}{10cm}
\begin{tabbing}
  {\bf Algorithm} {\scshape HenselLifting} ( $A$, $b$, $p$, $e_1\leq\ldots\leq e_m$ )\\[0.5ex]
  ~~\=Determine the solutions to $Ax \equiv b \pmod p$.\\
  \>{\bf for} $\ell\in\{ 1,\ldots, e_m\}$ {\bf do} \\
  \>\quad\= Echelonize the sub-matrix $\ti A$ modulo $p^{\ell}$.\\[0.2ex]
  \>\> Let $k(\ell) = \min\{ i \mid \ell\leq e_i \}$.\\
  \>\> {\bf for} $n \in \{ m, m-1, \ldots, k(\ell) \}$ {\bf do} \\
  \>\>\quad\= Let $\alpha_{jn}$ be the corner-entry corresponding to $x_n$. \\
  \>\>\> {\bf if} $\alpha_{jn} \not\equiv 0 \pmod p$ {\bf then} \\
  \>\>\>\quad\= Evaluate Equation~(\ref{eqn:UniqLift}) and
                determine $t\in\{0,\ldots,p-1\}$\\
  \>\>\>\> Lift the solution $x_n$ uniquely.\\
  \>\>\> {\bf else}\\
  \>\>\>\> Consider the condition~(\ref{eqn:CondLift}).\\
  \>\>\>\> {\bf if} this cannot be satisfied {\bf then}\\
  \>\>\>\>\quad\= {\bf return fail};\\
  \>\>\>\> {\bf else}\\
  \>\>\>\>\> Restrict to the solutions in Lemma~\ref{lem:NonUni}.\\
  \>\>\>\>\> Lift these latter solutions.\\
  \>\>\>\> {\bf end if};\\
  \>\>\> {\bf end if};\\
  \>\> {\bf end for};\\
  \>{\bf end for};
\end{tabbing}
\end{minipage}
\caption{Solving endomorphic equations with the Hensel lemma}
\end{center}
\end{algorithm}
Note that we only used the fact that the linear system arises from
an endomorphism of an abelian $p$-group, in form of the divisibility
condition in~(\ref{eqn:DivCond}). Clearly, the divisibility condition
is always satisfied for any linear system of equations over a homocyclic
$p$-group. Therefore the algorithm {\scshape HenselLifting} also applies
to the following problems:
\begin{itemize}\addtolength{\itemsep}{-1ex}
\item Let $G$ be an abelian $p$-group and let $g_1,\ldots,g_n\in G$. 
      Decide whether or not $b \in \la g_1,\ldots,g_n\ra$ holds.
\item Let $G$ be an abelian $p$-group and let $\varphi\colon G^m \to G^n$
      be a homomorphism. Then the algorithm {\scshape HenselLifting}
      also solves the linear system of modular equations $Ax=b$, as the
      divisibility condition is also satisfied in this case.
\item Let $G$ be an abelian $p$-group with generators $g_1,\ldots,g_n$. 
      Further let $\varphi\colon G\to G$ be an automorphism given by the
      images $g_1^\varphi,\ldots,g_n^\varphi$.  Then the algorithm
      {\scshape HenselLifting} easily modifies so that it computes the
      inverse $\varphi^{-1}$.
\end{itemize}
The first problem yields that we can modify the algorithm in
Section~\ref{sec:Block} for solving a linear system $x^\varphi = b$
for an arbitrary homomorphism $\varphi\colon G\to H$ of two abelian
$p$-groups $G$ and $H$.

\subsubsection{Solving linear equations over arbitrary finite $p$-groups}
Let $G$ and $H$ be arbitrary abelian $p$-groups and let $\varphi\colon
G\to H$ be a homomorphism. The algorithm of Section~\ref{sec:HenselEndo}
generalizes to an algorithm for solving the linear system $x^\varphi
= b$ for some $b\in H$. This is a slight modification of the algorithm in
Section~\ref{sec:Block}. 

\nocite{GAP4}
\bibliographystyle{abbrv}
\bibliography{solv}

\end{document}